\documentclass[a4paper,11pt,oneside]{article}

\pdfoutput=1

\usepackage{hyperref}
\usepackage[numbers]{natbib}
\usepackage{amsmath,amsthm,amsfonts,amssymb,amscd}
\usepackage{mathtools}
\usepackage{pgf,tikz}
\usepackage{float}
\usepackage{enumerate}
\usepackage{booktabs}
\usepackage{pgfplots}
\usepackage[noend]{algpseudocode}
\usepackage{graphicx,bm,xcolor}
\usepackage{algorithm}
\usepackage{algpseudocode}
\usepackage{caption}
\usepackage[T1]{fontenc}
\usepackage{t1enc}
\usepackage[polish,german,english]{babel}
\usepackage{verbatim}
\usepackage{bbm} 
\usepackage{subfig} 
\usepackage{svg} 

\usepackage{url}
\usepackage{authblk}

\usepackage[top=2cm, bottom=2cm, left=2cm, right=2cm]{geometry}

\usetikzlibrary{arrows, shapes, shapes.misc, fit, calc}

\pgfdeclarelayer{bg}    
\pgfsetlayers{bg,main}  

\theoremstyle{plain} 
\newtheorem{theorem}{Theorem}[section]

\newtheorem{lemma}[theorem]{Lemma}

\newtheorem{remark}[theorem]{Remark}

\theoremstyle{definition} %

\theoremstyle{remark} %

\newcounter{numctr}

\DeclareMathOperator*{\argmin}{arg\,min}

\newcommand{\julian}[1]{{{\textcolor{green}{#1}}}}
\newcommand{\maxs}[1]{{{\textcolor{blue}{#1}}}}

\begin{document}

\title{Neural network guided adjoint computations in dual weighted residual error estimation}
\author[1]{J. Roth}
\author[1]{M. Schr\"oder}
\author[1,2]{T. Wick}

\affil[1]{Leibniz Universit\"at Hannover, Institut f\"ur Angewandte
  Mathematik, AG Wissenschaftliches Rechnen, Welfengarten 1, 30167 Hannover, Germany}

\affil[2]{Cluster of Excellence PhoenixD (Photonics, Optics, and
	Engineering - Innovation Across Disciplines), Leibniz Universit\"at Hannover, Germany}

\date{}

\maketitle
	
\begin{abstract}
In this work, we are concerned with neural network guided goal-oriented a posteriori error estimation and 
adaptivity using the dual weighted residual method. 
The primal problem is solved 
using classical Galerkin finite elements. The adjoint problem is solved 
in strong form with a feedforward neural network using two or three hidden layers. 
The main objective of our approach is to explore alternatives for solving the adjoint problem 
with greater potential of a numerical cost reduction.
The proposed algorithm is based on the general goal-oriented error estimation 
theorem including both linear and nonlinear stationary partial differential equations and goal functionals. 
Our developments are substantiated with some numerical experiments that include comparisons 
of neural network computed adjoints and classical finite element solutions of the adjoints. 
In the programming software, the open-source library deal.II is successfully coupled 
with LibTorch, the PyTorch C++ application programming interface.
\end{abstract}

\section{Introduction}
This work is devoted to an innovative solution of the adjoint equation in goal-oriented 
error estimation with the dual weighted residual (DWR) method \cite{BeRa96,becker_rannacher_2001,bangerth_rannacher_2003} (based on former adjoint concepts \cite{ErikssonEstepHansboJohnson1995}); we also refer to \cite{braack_ern_2003,AinsworthOden:2000,GilesSuli2002,PerairePatera1998}
for some important early work. Since then, the DWR method has been applied to numerous applications such as fluid-structure interaction \cite{ZeeBrummelenAkkermanBorst2011,Ri12_dwr,FaiWi18},
Maxwell's equations \cite{CiBo20_maxwell}, surrogate models in stochastic inversion \cite{MaWo18},
model adaptivity in multiscale problems \cite{MaiRa18},
and adaptive multiscale predictive modeling \cite{Od18}. A summary of theoretical advancements 
in efficiency estimates and multi-goal-oriented error estimation was recently made in \cite{Endt21}.
An important part in these studies is the adjoint 
problem, as it measures the sensitivity of the primal 
solution with respect to a single or multiple given goal functionals (quantities of interest). This adjoint solution 
is usually obtained by global higher order finite element method (FEM) solutions or local higher order approximations
\cite{becker_rannacher_2001}. In general, the former is more stable, see e.g. \cite{EndtLaWi20_smart}, but 
the latter works often sufficiently well in practice. 
As the adjoint solution is only required to evaluate 
the a posteriori error estimator, a cheap solution is of interest. 

Consequently, in this work, the main objective is to explore alternatives 
for computing the adjoint.
Due to the universal approximation property \cite{pinkus_1999}, a primer candidate are neural networks as they are already successfully employed for solving 
partial differential equations (PDE) \cite{raissi2017pinns,sirignano2017dgm,Berg_2018,raissi_hiddenfluid,li2020fourier, WESSELS2020,hartmann2020neural,SAMANIEGO2020112790,hennigh2020nvidia}. 
A related work in aiming to improve goal-oriented computations with the help of neural network data-driven
finite elements is \cite{BREVIS2020}.
Moreover, a recent summary of the key concepts of neural networks and deep learning was compiled in \cite{HiHi19}.
The advantage of neural networks is a greater flexibility as they belong to the class of meshless methods. We follow the methodology of \cite{raissi2017pinns, sirignano2017dgm} to solve PDEs by minimizing the residual using 
an L-BFGS (Limited memory Broyden-Fletcher-Goldfarb-Shanno) method \cite{Liu1989}.
We address both linear and nonlinear PDEs and goal functionals in stationary settings.
However, a shortcoming in the current approach is that we need to work with strong adjoint formulations,
which may limit extensions to nonlinear coupled PDEs such as multiphysics problems and coupled variational 
inequality systems. If such problems can be restated in an energy formulation, again neural network algorithms
are known \cite{E2018, SAMANIEGO2020112790}. Despite this drawback, namely the necessity of working with strong formulations, 
the current study provides useful insights whether 
at all neural network guided adjoints can be an alternative concept for dual weighted residual error estimation. 
For this reason, our resulting modified adaptive algorithm and related numerical simulations
are compared side by side in all numerical tests to classical Galerkin finite element solutions (see e.g., \cite{ciarlet2002fem}) of the adjoint.
Our proposed algorithm is implemented in the open-source finite element library 
deal.II \cite{dealII92} coupled with LibTorch, the PyTorch C++ API \cite{paszke2019pytorch}.

The outline of this paper is as follows: In Section \ref{sec:DWR}, we recapitulate the DWR method.
Next, in Section \ref{sec:NN} we gather the important ingredients of the neural network solution. This section
also includes an extension of an approximation theorem from Lebesgue spaces to classical function spaces.
The algorithmic realization is addressed in Section \ref{sec:algo}. Then, in Section \ref{sec:tests}
several numerical experiments are conducted. Our findings are summarized in Section \ref{sec:conclusions}.

\section{Dual weighted residual method}
\label{sec:DWR}

\subsection{Abstract problem}\label{sec:abstract_problem}
Let $U$ and $V$ be Banach spaces and let $\mathcal{A}: U \rightarrow V^\ast$ be a nonlinear mapping, where $V^\ast$ denotes the dual space of $V$. With this, we can define the problem: Find $u \in U$ such that 
\begin{align}\refstepcounter{numctr}\label{primalproblem}
    \mathcal{A}(u)(v) = 0 \quad \forall v \in V. \tag{\thenumctr}
\end{align}
Additionally, we can look at an approximation of this problem. For subspaces $\tilde{U} \subset U$ and $\tilde{V} \subset V$ the problem reads: Find $\tilde{u} \in \tilde{U}$ such that
\begin{align*}
    \mathcal{A}(\tilde{u})(\tilde{v}) = 0 \quad \forall \tilde{v} \in \tilde{V}.
\end{align*}
 \begin{remark}
      In the following the nonlinear mapping $\mathcal{A}(\cdot)(\cdot)$ will represent the variational formulation of a stationary partial differential equation with the associated function spaces $U$ and $V$. We define the finite element approximation of the abstract problem as follows: Find $u_h \in U_h$ such that
\begin{align}\refstepcounter{numctr}\label{discrete_primalproblem}
    \mathcal{A}(u_h)(v_h) = 0 \quad \forall v_h \in V_h, \tag{\thenumctr}
\end{align}
where $U_h \subset U$ and $V_h \subset V$ denote the finite element spaces. Here the operator is given by $\mathcal{A}(u_h)(\cdot) := a(u_h)(\cdot) - l(\cdot)$ with the linear forms $a(u_h)(\cdot)$ and $l(\cdot)$.
 \end{remark}

\subsection{Motivation for adaptivity}   

In many applications we are not necessarily interested in the whole solution to a given problem but more explicitly only in the evaluation of a certain quantity of interest. This quantity of interest can often be represented mathematically by a goal functional $J: U \rightarrow \mathbb{R}$. Here the main target is to minimize the error in this given goal functional and use the computational resources efficiently. This can lead to the approach of \cite{BeRa96, becker_rannacher_2001}, the DWR method, which this work will follow closely. We are interested in the evaluation of the goal functional $J$ in the solution $u \in U$ to the problem $\mathcal{A}(u)(v) = 0$ for all $v \in V$. Under the assumption that the problem yields a unique solution, the formulation from above can be rewritten into the equivalent optimization problem
\begin{align*}
    \min_{u \in U} J(u) \quad s.t. \quad \mathcal{A}(u)(v) = 0 \ \forall v \in V.
\end{align*}
For this constrained optimization problem we can introduce the corresponding Lagrangian 
\begin{align*}
    \mathcal{L}(u,z) = J(u) - \mathcal{A}(u)(z)
\end{align*}
with the adjoint variable $z \in V$. For this, a stationary point needs to fulfill the first-order necessary conditions
\begin{align*}
    \mathcal{L}^\prime = 0 &\Leftrightarrow 
    \begin{cases}
      \mathcal{L}_u^\prime(u,z) = J^\prime(u)(\delta u) - \mathcal{A}^\prime(u)(\delta u,z) &\stackrel{!}{=}0\\
      \mathcal{L}_z^\prime(u,z) = -\mathcal{A}(u)(\delta z) &\stackrel{!}{=}0
    \end{cases} \\
    &\Leftrightarrow
    \begin{cases}
      \mathcal{A}^\prime(u)(\delta u,z) = J^\prime(u)(\delta u)\\
      \mathcal{A}(u)(\delta z) = 0
    \end{cases}
\end{align*}
where $J^\prime , \mathcal{A}^\prime$ denote the Fr\'echet derivatives. We see that a defining equation for the adjoint variable arises therein. Find $z \in V$ such that
\begin{align}\refstepcounter{numctr}\label{adjointproblem}
    \mathcal{A}^\prime(u)(\phi,z) = J^\prime(u)(\phi) \quad \forall \phi \in U, \tag{\thenumctr}
\end{align}
which is known as the adjoint problem. This leads to the error representation for arbitrary approximations, as derived in \cite{rannachervihharev_balancing}.
\begin{theorem} \label{erroestimator}
Let $(u,z) \in U \times V$ solve (\ref{primalproblem}) and (\ref{adjointproblem}). Further, let $\mathcal{A} \in \mathcal{C}^3(U,V^\ast) $ and $J \in \mathcal{C}^3(U,\mathbb{R})$.
Then for arbitrary approximations $(\Tilde{u},\Tilde{z}) \in U \times V$ the error representation
\begin{align*}\refstepcounter{numctr}
    J(u)-J(\Tilde{u}) = \frac{1}{2}\rho (\Tilde{u})(z-\Tilde{z})+\frac{1}{2}\rho^*(\Tilde{u},\Tilde{z})(u-\Tilde{u}) + \rho (\Tilde{u})(\Tilde{z}) +  \mathcal{R}^{(3)} \tag{\thenumctr}
\end{align*}
holds true and 
\begin{align*}
    \rho (\Tilde{u})(\cdot) &\coloneqq -\mathcal{A}(\Tilde{u})(\cdot), \\ 
    \rho^*(\Tilde{u},\Tilde{z})(\cdot) &\coloneqq J^{\prime}(\Tilde{u})(\cdot) - \mathcal{A}^{\prime}(\Tilde{u})(\cdot,\Tilde{z}).
\end{align*}
With $e = u-\Tilde{u}, \ e^* = z - \Tilde{z}$, the remainder term reads as follows:
\begin{align*}
    \mathcal{R}^{(3)} \coloneqq \frac{1}{2}\int_0^1 \Big[J^{\prime \prime \prime}(\Tilde{u}+se)(e,e,e) - \mathcal{A}^{\prime \prime \prime}(\Tilde{u}+se)(e,e,e,\Tilde{z}+se^*) \nonumber\\ -3\mathcal{A}^{\prime \prime}(\Tilde{u}+se)(e,e,e^*)\Big]s(s-1) \ \mathrm{d}s.
\end{align*}
\end{theorem}
\begin{proof}
The proof can be found in \cite{rannachervihharev_balancing}.
\end{proof}
\begin{remark}
If $\tilde{u} := u_h \in U_h \subset U$ is the Galerkin projection which solves (\ref{discrete_primalproblem}) and $\tilde{z} := z_h \in V_h \subset V $, then the iteration error $\rho ( \tilde{u}) (\tilde{z})$ vanishes and yields the theorems presented in the early work \cite{bangerth_rannacher_2003}.
Therefore, from now on we omit the iteration error.  The remainder term is usually of third order \cite{becker_rannacher_2001} and can be omitted for which detailed computational evidence was 
demonstrated in \cite{endtmayer2020reliability}.
In the case of a linear problem, it clearly holds that 
\begin{align*}
    \eta = \rho (\Tilde{u})(z-\Tilde{z}) = \frac{1}{2}\rho (\Tilde{u})(z-\Tilde{z})+\frac{1}{2}\rho^*(\Tilde{u},\Tilde{z})(u-\Tilde{u}).
\end{align*}
\end{remark}

\begin{remark}
     Theorem \ref{erroestimator} motivates the error estimator
     \begin{align*}
         \eta = \frac{1}{2}\rho (\Tilde{u})(z-\Tilde{z})+\frac{1}{2}\rho^*(\Tilde{u},\Tilde{z})(u-\Tilde{u}).
     \end{align*}
     This error estimator is exact but not computable. Therefore, the exact solutions $u$ and $z$ are now being approximated by higher-order solutions $\left(u_h^{(2)},z_h^{(2)}\right) \in U_h^{(2)} \times V_h^{(2)}$. These higher-order solutions can be realised by a globally refined grid or by using higher-order basis functions. The practical error estimator reads
     \begin{align*}\refstepcounter{numctr}\label{comp_erroresti}
         \eta^{(2)}= \frac{1}{2}\rho (\Tilde{u})\left(z_h^{(2)}-\Tilde{z}\right)+\frac{1}{2}\rho^*(\Tilde{u},\Tilde{z})\left(u_h^{(2)}-\Tilde{u}\right). \tag{\thenumctr}
     \end{align*}
\end{remark}

\subsection{DWR Algorithm}
\noindent In principle, we need to solve four problems, where especially the computation of  $u_h^{(2)}$ is expensive. 
It is well-known that different possibilities exist such as global higher-order finite element solution 
or local interpolations \cite{becker_rannacher_2001,RiWi15_dwr,braack_ern_2003}.
Moreover, we only consider the primal part of the error estimator, which is justified 
for linear problems only, and yields a second order remainder term in nonlinear problems \cite{becker_rannacher_2001}[Proposition 2.3]:
\begin{align*}
    \eta_h^{(2)} = \rho(u_h)\left(z_h^{(2)}-\tilde{z}\right).
\end{align*}
For many nonlinear problems this version is used as it reduces to solving only two problems
and yields for mildly nonlinear problems, such as incompressible flow \cite{BrRi06}, excellent values.
On the other hand, for quasi-linear problems, there is a strong need to work with the adjoint 
error parts $\rho^*$ as well \cite{EndtLaWi18,endtmayer2020reliability}.

In our work, we employ solutions in enriched spaces. We compute the adjoint solution $z_h^l = i_h z_h^{l,(2)} \in V_h^l \subset V_h^{l,(2)}$ via restriction. For nonlinear problems, we approximate the primal solution in the enriched space $u_h^{l,(2)} = I_h^{(2)}u_h^l \in U_h^{l,(2)} \supset U_h^l$ via interpolation. Therefore, we only solve two problems in practice: the primal problem and the enriched adjoint problem.

\begin{algorithm}[H]
\caption{DWR algorithm for general nonlinear problems} \label{dwr_algo}
\begin{algorithmic}[1]
\State Start with some initial guess $u_h^0, l = 1$.
\State Solve the primal problem: Find $u_h^l \in U_h^l$ such that 
\begin{align*}
    \mathcal{A}\left(u_h^l\right)\left(\phi_h^l\right) = 0 \quad \forall \phi_h^l \in V_h^l ,
\end{align*}
using some nonlinear solver.
\State Compute the interpolations $u_h^{l,(2)} = I_h^{(2)}u_h^l \in U_h^{l,(2)}$.
\State Solve the enriched adjoint problem: Find $z_h^{l,(2)} \in V_h^{l,(2)}$ such that
\begin{align*}
    \mathcal{A}'\left(u_h^{l,(2)}\right)\left(z_h^{l,(2)},\psi_h^{l,(2)}\right) = J'\left(u_h^{l,(2)}\right)\left(\psi_h^{l,(2)}\right) \quad \forall \psi_h^{l,(2)} \in U_h^{l,(2)},
\end{align*}
using some linear solver.
\State Compute the restriction $z_h^l = i_h z_h^{l,(2)} \in V_h^l$.
\State Compute the error estimator $\eta^{(2)}$.

\If{$\left| \eta^{(2)} \right| < TOL$}
    \State Algorithm terminates with final output $J\left(u_h^l\right)$.
\EndIf
\State Localize error estimator $\eta^{(2)}$ and mark elements.
\State Refine marked elements: $\mathbb{T}_h^l \mapsto \mathbb{T}_h^{l+1}, l = l+1.$
\State Go to Step 2.  
\end{algorithmic}
\end{algorithm}

\subsection{Error localization}
The error estimator $\eta^{(2)}$ must be localized to corresponding regions of error contribution. This can be either done by methods proposed in \cite{BeRa96,becker_rannacher_2001,bangerth_rannacher_2003}, which use integration by parts in a backwards manner and result in an element wise localization employing the strong 
form of the equations.
However, in this work we use the technique of \cite{RiWi15_dwr}, where a partition-of-unity (PU) $\sum_i \psi_i \equiv 1$ was introduced, in which the error contribution is localized on a nodal level. To realize this partition-of-unity, one can simply choose piece-wise bilinear elements $\lbrace \psi_h^i\, |\, i = 1,\dots,N \rbrace$. Then, the approximated error indicator reads
\begin{align*} \refstepcounter{numctr}\label{pu_dwr_estimator}
    \eta^{(2), PU} = \sum_{i=1}^{N} \left( \frac{1}{2}\rho (\Tilde{u})\left(\left(z_h^{(2)}-\Tilde{z}\right)\psi_i\right)+\frac{1}{2}\rho^*(\Tilde{u},\Tilde{z})\left(\left(u_h^{(2)}-\Tilde{u}\right)\psi_i\right) \right). \tag{\thenumctr}
\end{align*}
Some recent theoretical work on the effectivity and efficiency of $\eta^{(2), PU}$ can be found in \cite{RiWi15_dwr,endtmayer2020reliability}, respectively. The main objective of the remainder of this paper is to compute the adjoint solution with a feedforward neural network.

\subsection{Effectivity index}
To evaluate the goodness of the error estimator we introduce the effectivity index 
\begin{align*}
    I_{eff} = \frac{\left|\eta^{(2), PU}\right|}{|J(u)-J(\Tilde{u})|}.
\end{align*}
If $J(u)$ is unknown, we approximate it by $J(\hat{u})$, where $\hat{u}$ is the solution of the PDE on a very fine grid. We desire that the effectivity index converges to $1$, which signifies that our error estimator is a good approximation of the error in the goal functional.

\section{Neural networks}
\label{sec:NN}
In order to realize neural network guided DWR, we consider feedforward neural networks $u_{NN}: \mathbb{R}^d \rightarrow \mathbb{R}$, where $d$ is the dimension of the domain $\Omega$ plus the dimension of $u$ and the dimension of all the derivatives of $u$ that are required for the adjoint problem.
The neural networks can be expressed as
\begin{align*}
    u_{NN}(x) = T^{(L)} \circ \sigma \circ T^{(L-1)} \circ \cdots \circ \sigma \circ T^{(1)}(x),
\end{align*}
where $T^{(i)}: \mathbb{R}^{n_{i-1}} \rightarrow \mathbb{R}^{n_{i}}, y \mapsto  W^{(i)}y\, +\, b^{(i)}$ are affine transformations for $1 \leq i \leq L$, with weight matrices $W^{(i)} \in \mathbb{R}^{n_i \times n_{i-1}}$ and bias vectors $b^{(i)} \in \mathbb{R}^{n_i}$. Here $n_{i}$ denotes the number of neurons in the $i$.th layer with $n_0 = d$ and $n_L = 1$. $\sigma: \mathbb{R} \rightarrow \mathbb{R}$ is a nonlinear activation function, which is the hyperbolic tangent function throughout this work. Derivatives of neural networks can be computed with back propagation (see e.g. \cite{RuHiWi86,HiHi19}), a special case of reverse mode automatic differentiation \cite{Nocedal2006}.  Similarly higher order derivatives can be calculated by applying automatic differentiation recursively. 

\subsection{Universal function approximators}
Cybenko \cite{Cybenko1989} and Hornik \cite{HORNIK1991} proved a first version of the universal approximation theorem, which states that continuous functions can be approximated to arbitrary precision by single hidden layer neural networks. A few years later Pinkus \cite{pinkus_1999} generalized their findings and showed that single hidden layer neural networks can uniformly approximate a function and its partial derivatives. The space of single hidden layer neural networks is given by
\begin{align*}
    \mathcal{M}(\sigma) := \span \left\lbrace \sigma(w \cdot x + b) \;\middle\vert\; w \in \mathbb{R}^d, b \in \mathbb{R} \right\rbrace.
\end{align*}
\begin{theorem}[Universal Approximation Theorem \cite{pinkus_1999}]
    Let $m^{(i)} \in \mathbb{N}_0^d$ be multi indices for $1 \leq i \leq s$ and set $m = \underset{1 \leq i \leq s}{\max} |m^{(i)}|$. Assume $\sigma \in C^m(\mathbb{R})$ and $\sigma$ is not a polynomial. Then for any $f \in \cap_{i = 1}^s C^{m^{(i)}}(\mathbb{R}^d)$, any compact $K \subset \mathbb{R}^d$, and any $\epsilon > 0$, there exists $g \in \mathcal{M}(\sigma)$ such that
    \begin{align*}
        \underset{x \in K}{\max} \left| D^k f(x) - D^k g(x)  \right| < \epsilon,
    \end{align*}
    for all $k \in \mathbb{N}_0^d$ for which $k \leq m^{(i)}$ for some $1 \leq i \leq s$.
\end{theorem}
\noindent This theoretical result motivates the application of neural networks for the numerical approximation of partial differential equations.

\subsection{Residual minimization with neural networks}
Residual minimization with neural networks has become popular in the last few years by the works of Raissi, Perdikaris and Karniadakis on physics-informed neural networks (PINNs) \cite{raissi2017pinns} and the paper of Sirignano and Spiliopoulos on the "Deep Galerkin Method" \cite{sirignano2017dgm}. For their approach one can consider the strong formulation of the stationary PDE
\begin{align*}\refstepcounter{numctr}\label{strong_form}
\begin{split}
    \mathcal{N}(u,x) &= 0 \quad \text{in}\ \Omega \\
    \mathcal{B}(u,x) &= 0 \quad \text{on}\ \partial\Omega 
\end{split} \tag{\thenumctr}
\end{align*}
where $\mathcal{N}$ is a differential operator and $\mathcal{B}$ is a boundary operator.
An example for the differential operator $\mathcal{N}$ is given by the semi-linear form $\mathcal{A}(u)(v)$ introduced in Section \ref{sec:abstract_problem}. The boundary operator $\mathcal{B}$ in case of Dirichlet conditions is realized in the weak formulation as usual in the function space $U$. One then needs to find a neural network $u_{NN}$, which minimizes the loss function
\begin{align*}
     L(u_{NN}) = \frac{1}{n_\Omega}  \sum_{i = 1}^{n_\Omega} \mathcal{N}\left(u_{NN},x_i^\Omega\right)^2 + \frac{1}{n_{\partial\Omega}}  \sum_{i = 1}^{n_{\partial\Omega}} \mathcal{B}\left(u_{NN},x_i^{\partial\Omega}\right)^2, 
\end{align*}
where $x_1^\Omega, \dots, x_{n_\Omega}^\Omega \in \Omega$ are collocation points inside the domain and $x_1^{\partial\Omega}, \dots, x_{n_{\partial\Omega}}^{\partial \Omega} \in \partial\Omega$ are collocation points on the boundary. In \cite{vandermeer2020optimally} it has been shown that the two components of the loss function need to be weighted appropriately to yield accurate results. Therefore, we use a modified version of this method which circumvents these issues.

\subsection{Our approach} \label{approach}
Let us again consider the abstract PDE problem in its strong formulation (\ref{strong_form}).
For simplicity, we only consider Dirichlet boundary conditions, i.e. $\mathcal{B}(u,x) := u(x) - g(x)$.
Additionally, in our work we use the approach of Berg and Nystr\"om \cite{Berg_2018}, who used the ansatz
\begin{align*}\refstepcounter{numctr}\label{ansatz}
    u(x) := d_{\partial \Omega}(x)\cdot u_{NN}(x) + \tilde{g}(x) \quad \text{for}\, x \in \bar{\Omega} \tag{\thenumctr} 
\end{align*} 
to fulfill inhomogeneous Dirichlet boundary conditions exactly. Here $\tilde{g}$ denotes the extension of the boundary data $g$ to the entire domain $\bar{\Omega}$, which is continuously differentiable up to the order of the differential operator $\mathcal{N}$. Berg and Nystr\"om \cite{Berg_2018} used the distance to the boundary $\partial \Omega$ as their function $d_{\partial \Omega}$. However, it is sufficient to use a function $d_{\partial \Omega}$ which is continuously differentiable up to the order of the differential operator $\mathcal{N}$ with the properties
\begin{align*}
    d_{\partial \Omega}(x) \begin{cases}
    = 0 & \text{for } x \in \partial \Omega \\
    \neq 0 & \text{for } x \in \Omega
    \end{cases}.
\end{align*}
Thus, $d_{\partial \Omega}$ can be interpreted as a level-set function, since
\begin{align*}
    \Omega = \left\lbrace x \in \bar{\Omega} \;\middle\vert\; d_{\partial \Omega}(x) \neq 0 \right \rbrace \text{ and } \partial\Omega = \left \lbrace x \in \bar{\Omega} \;\middle\vert\; d_{\partial \Omega}(x) = 0 \right \rbrace .
\end{align*}
Obviously, for this kind of ansatz for the solution of the PDE, it holds that
\begin{align*}
\mathcal{B}(u,x) = u(x) - g(x) = \left[d_{\partial \Omega}(x)\cdot u_{NN}(x) + \tilde{g}(x) \right] - g(x) = 0 \quad  \text{on}\ \partial\Omega.
\end{align*}
Therefore, in contrast to some previous works, we do not need to account for the boundary conditions in our loss function, which is a big benefit of our approach, since proper weighting of the different residual contributions in the loss function is not required. It might only be a little cumbersome to fulfill the boundary conditions exactly when dealing with mixed boundary condition, but the form of the ansatz function for such boundary conditions has been laid out in \cite{lyu2020enforcing}.

\def\layersep{2.5cm}
\definecolor{dark-blue}{RGB}{0,0,255}

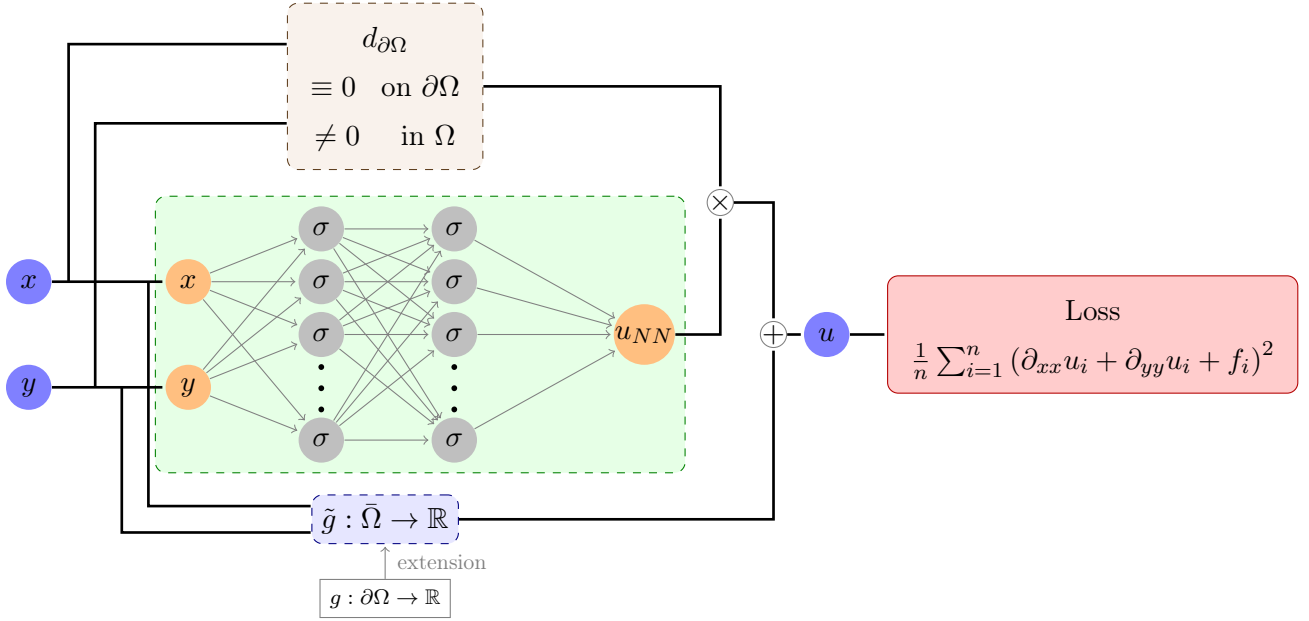
\begin{figure}[H]
\begin{center}
\begin{tikzpicture}[scale = 0.7, shorten >=1pt,->,draw=black!50, node distance=\layersep]
    \tikzstyle{neuron}=[circle,fill=black!25,minimum size=17pt,inner sep=0pt]
    \tikzstyle{input neuron}=[neuron, fill=orange!50];
    \tikzstyle{output neuron}=[neuron, fill=orange!50];
    \tikzstyle{hidden neuron}=[neuron, fill=gray!50];
    \tikzstyle{neuron missing}=[draw=none, scale=2, text height=0.333cm, execute at begin node=\color{black}$\vdots$, minimum size=7pt];
    \tikzstyle{operator} = [circle, draw, inner sep=-0.5pt, minimum height = 0.2cm]
    
    \tikzstyle{coordinate}=[neuron, fill=dark-blue!50];

    
    \node[coordinate] (X) at (-3,-1.5) {$x$};
    \node[coordinate] (Y) at (-3,-3.5) {$y$};

    
    \node[input neuron] (I-1) at (0,-1.5) {$x$};
    \node[input neuron] (I-2) at (0,-3.5) {$y$};
    
    \draw[-, black, line width = 1pt](X) -- (I-1);
    \draw[-, black, line width = 1pt](Y) -- (I-2);

    \foreach \name / \y in {1,2,3,5}
        \path[yshift=0.5cm]
            node[hidden neuron] (H_ONE-\name) at (\layersep,-\y cm) {$\sigma$};
    \node[neuron missing] (H_ONE-4) at (\layersep,-3.5 cm) {};
            
    \foreach \name / \y in {1,2,3,5}
        \path[yshift=0.5cm]
            node[hidden neuron] (H_TWO-\name) at (\layersep + \layersep,-\y cm) {$\sigma$};
     \node[neuron missing] (H_TWO-4) at (\layersep + \layersep,-3.5 cm) {};

    \node[output neuron, right of=H_TWO-3] (O) {$u_{NN}$};

    \foreach \source in {1,2}
        \foreach \dest in {1,2,3,5}
            \path (I-\source) edge (H_ONE-\dest);
            
    \foreach \source in {1,2,3,5}
        \foreach \dest in {1,2,3,5}
            \path (H_ONE-\source) edge (H_TWO-\dest);

    \foreach \source in {1,2,3,5}
        \path (H_TWO-\source) edge (O);
    
    \begin{pgfonlayer}{bg}    
        \node[draw,green!50!black, fill = green!10 , dashed, rounded corners,fit=(I-1) (I-2) (H_ONE-1) (H_ONE-5) (H_TWO-1) (H_TWO-5) (O)] {};
    \end{pgfonlayer}
    
    \node [draw, brown!50!black, dashed, rounded corners, fill=brown!10]  (D) at (\layersep + 1.2cm, 2.2cm) {$\begin{array}{c}{\color{black} d_{\partial\Omega}}\\ {\color{black}\equiv 0 \phantom{a} \mbox{ on } \partial\Omega} \\ {\color{black}\neq 0 \phantom{aa} \mbox{ in } \Omega} \end{array}$};
    
    \draw[-, black, line width = 1pt] (X)+(0.75cm,0cm) -- ++(0.75cm,4.5cm) -- ++(4.15cm, 0cm);
    \draw[-, black, line width = 1pt] (Y)+(1.25cm,0cm) -- ++(1.25cm,5cm) -- ++(3.65cm, 0cm);
    
    
    \node [draw, blue!50!black, dashed, rounded corners, fill=blue!10]  (B) at (\layersep + 1.2cm, -6cm) {${\color{black}\tilde{g} : \bar\Omega \rightarrow \mathbb{R}}$};
    
    \draw[-, black, line width = 1pt] (Y)+(1.75cm,0cm) -- ++(1.75cm,-2.75cm) -- ++(3.6cm, 0cm);
    \draw[-, black, line width = 1pt] (X)+(2.25cm,0cm) -- ++(2.25cm,-4.25cm) -- ++(3.1cm, 0cm);

    \node [draw]  (oldB) at (\layersep + 1.2cm, -7.5cm) {\scriptsize$g : \partial\Omega \rightarrow \mathbb{R}$};
    
    \draw (oldB) -- node [gray, scale= 0.75, xshift = 1cm]{extension} (B);
    
    \node [operator] (mult) at (10,0) {$\times$};
    \draw[-, black, line width = 1pt] (D) to (D -| mult) to (mult);
    \draw[-, black, line width = 1pt] (O) to (O -| mult) to (mult);
    
    \node [operator] (add) at (11,-2.5) {$+$};
    \draw[-, black, line width = 1pt] (mult) to (mult -| add) to (add);
    \draw[-, black, line width = 1pt] (B) to (B -| add) to (add);
    
    \node[neuron, fill=dark-blue!50] (U) at (12,-2.5) {$u$};
    \draw[-, black, line width = 1pt] (add) to (U);
    
    \node [draw, red!70!black, rounded corners, fill=red!20]  (loss) at (17, -2.5) {$\begin{array}{c} {\color{black}\mbox{Loss}}\\ {\color{black}\frac{1}{n}\sum_{i=1}^n\left(\partial_{xx}u_i + \partial_{yy}u_i + f_i \right)^2}\end{array}$};
    \draw[-, black, line width = 1pt] (U) to (loss);
\end{tikzpicture}
\caption{Section \thesubsection: Diagram of our ansatz $u = d_{\partial\Omega}\cdot u_{NN} + \tilde{g}$ for the two dimensional Poisson problem. Here we used the abbreviations $u_i := u(x_i, y_i)$ and $f_i := f(x_i,y_i)$.}
\end{center}
\end{figure}

 \subsubsection{Approximation theorem}
In the following, we prove that our neural network solutions approximate the analytical solutions well if their loss is sufficiently small.
Our neural networks $u_{NN}$ have been trained with the mean squared error of the residual of the PDE, i.e.
\begin{align*}\refstepcounter{numctr}\label{loss}
    L(u) = \frac{1}{n}  \sum_{i = 1}^n \mathcal{N}(u,x_i)^2, \tag{\thenumctr} 
\end{align*}
where $n$ is the number of collocation points $x_i$ from the domain $\Omega$. For the sake of generality, let us consider the generalized loss
\begin{align*}
    \hat{L}_p(u) = \frac{1}{|\Omega|} \int_\Omega |\mathcal{N}(u,x)|^p\ \mathrm{d}x
\end{align*}
for $p \geq 1$.
Then, the loss (\ref{loss}) is just the Monte Carlo approximation of the generalized loss for $p = 2$.
We briefly recall the approximation theorem from \cite{vandermeer2020optimally}
and show that the classical solution of the Poisson problem satisfies the assumptions of the approximation theorem.

\begin{lemma}[Approximation theorem \cite{vandermeer2020optimally}]\label{thm:approximation}
Let $2 \leq p \leq \infty$. We consider a PDE of the form (\ref{strong_form}) on a bounded, open domain $\Omega \subset \mathbb{R}^m$ with Lipschitz boundary $\partial \Omega$ and $\mathcal{N}(u,x) = N(u,x) - \hat{f}(x)$, where $N$ is a linear, elliptic operator and $\hat{f} \in L^2(\Omega)$. Let there be a unique solution $\hat{u} \in H^1(\Omega)$ and 
let the following stability estimate
\begin{align*}
    \| u \|_{H^1(\Omega)} \leq C \| f  \|_{L^2(\Omega)}
\end{align*}
hold for $u \in H^1(\Omega), f \in L^2(\Omega)$ with $N(u,x) = f(x)$ in $\Omega$.
Then we have for an approximate solution $u\in H^1(\Omega)$ that
\begin{align*}
    \forall \epsilon > 0\, \exists \delta > 0: \quad \hat{L}_p(u) < \delta \Longrightarrow \|u - \hat{u} \|_{H^1(\Omega)} < \epsilon.
\end{align*}
\end{lemma}

\begin{proof}
Let
\begin{align*}
    \delta = \epsilon^p C^{-p} |\Omega|^{-\frac{p}{2}}.
\end{align*}
Let $u = d_{\partial \Omega}(x)\cdot u_{NN}(x) + \tilde{g}(x) \in H^1(\Omega)$
be an approximate solution of the PDE with $\hat{L}_p(u) < \delta$, which means that there exists a perturbation to the right-hand side $f_{\text{error}} \in L^2(\Omega)$ such that $N(u,x) = \hat{f}(x) + f_{\text{error}}(x)$. By the stability estimate and the linearity of $N$, we have
\begin{align*}
    \|u - \hat{u}\|_{H^1(\Omega)} \leq C \|(\hat{f} + f_\text{error}) - \hat{f} \|_{L^2(\Omega)} = C \|f_\text{error}\|_{L^2(\Omega)}.
\end{align*}
Applying the H\"older inequality to the norm of $f_\text{error}$ and using $2 \leq p \leq \infty$ yields
\begin{align*}
    \|f_\text{error}\|_{L^2(\Omega)} \leq |\Omega|^{\frac{1}{2} - \frac{1}{p}}\|f_\text{error}\|_{L^p(\Omega)}.
\end{align*}
Combing the last two inequalities gives us the desired error bound
\begin{align*}
    \|u - \hat{u}\|_{H^1(\Omega)} &\leq  C \|f_\text{error}\|_{L^2(\Omega)} \leq C |\Omega|^{\frac{1}{2} - \frac{1}{p}}\|f_\text{error}\|_{L^p(\Omega)} \\
    &= C |\Omega|^{\frac{1}{2}} \hat{L}_p(u)^{\frac{1}{p}} \\
    &< C |\Omega|^{\frac{1}{2}} \delta^{\frac{1}{p}} = \epsilon.
\end{align*}
In the last inequality, we used that the generalized loss of our approximate solution is sufficiently small, i.e. $\hat{L}_p(u) < \delta$. 
\end{proof}

\noindent
Let us recapitulate an important result from the Schauder theory \cite{gilbarg2001}, which yields the existence and uniqueness of classical solutions of the Poisson problem if we assume higher regularity of our problem, i.e. when we work with H\"older continuous functions and sufficiently smooth domains.

\begin{lemma}[Solution in classical function spaces]
\label{theo_classical_spaces}
    Let $0 < \lambda < 1$ be such that $\Omega  \subset \mathbb{R}^m$ is a domain with $C^{2,\lambda}$ boundary, $\tilde{g} \in C^{2,\lambda}(\bar{\Omega})$ and $\hat{f} \in C^{0,\lambda}(\bar{\Omega})$. Then Poisson's problem, which is of the form (\ref{strong_form}) with $N(u,x) := -\Delta u$, has a unique solution $\hat{u} \in C^{2,\lambda}(\bar{\Omega})$.
\end{lemma}
\begin{proof}
Follows immediately from \cite{gilbarg2001}[Theorem 6.14].
\end{proof}

\noindent
With Lemma \ref{theo_classical_spaces} we can now show that the approximation theorem holds for the Poisson problem in classical function spaces.
\begin{theorem}
 \label{theo_classical_solution}
   Let $0 < \lambda < 1$ be such that $\Omega \subset \mathbb{R}^m$ is a bounded, open domain with $C^{2,\lambda}$ boundary, $\tilde{g} \in C^{2,\lambda}(\bar{\Omega})$ and $\hat{f} \in C^{0,\lambda}(\bar{\Omega})$. Then Poisson's problem, which is of the form (\ref{strong_form}) with $N(u,x) := -\Delta u$, has a unique solution $\hat{u} \in H^1(\Omega)$.
   Furthermore, there exists
    $u = d_{\partial \Omega}(x)\cdot u_{NN}(x) + \tilde{g}(x) \in H^1(\Omega)$ with the estimate
\begin{align*}
    \forall \epsilon > 0\, \exists \delta > 0: \quad \hat{L}_p(u) < \delta \Longrightarrow \|u - \hat{u} \|_{H^1(\Omega)} < \epsilon.
\end{align*}
\end{theorem}
\begin{proof}
     From Lemma \ref{theo_classical_spaces} it follows that there exists a unique solution $\hat{u} \in C^{2,\lambda}(\bar{\Omega}) \subset H^1(\Omega)$. Analogously it holds that $u = d_{\partial \Omega}(x)\cdot u_{NN}(x) + \tilde{g}(x) \in H^1(\Omega)$. Furthermore, we have by the Lax-Milgram Lemma that $\hat{u} \in H^1(\Omega)$ is the unique weak solution and fulfills the stability estimate
     \begin{align*}
        \| \hat{u} \|_{H^1(\Omega)} \leq C \| \hat{f} \|_{L^2(\Omega)}.
     \end{align*}
     By Lemma \ref{thm:approximation} the estimate
     \begin{align*}
        \forall \epsilon > 0\, \exists \delta > 0: \quad \hat{L}_p(u) < \delta \Longrightarrow \|u - \hat{u} \|_{H^1(\Omega)} < \epsilon
     \end{align*}
     then also holds.
\end{proof}

\begin{remark}
Theorem \ref{theo_classical_solution} implies that a low loss value of a neural network with high probability corresponds to an accurate approximation $u$ of the exact solution $\hat{u}$ of the PDE, since the loss is a Monte Carlo approximation of the generalized loss, which for a large number of collocation points should be close in value.
\end{remark}

\subsubsection{Neural network solution of the adjoint PDE}
To make a posteriori error estimates for our FEM solution of the primal problem (\ref{primalproblem}), we now use neural networks to solve the adjoint PDE (\ref{adjointproblem}). In an FEM approach, the adjoint PDE would be solved in its variational form as described in Algorithm \ref{dwr_algo}, but we minimize the residual of the strong form using neural networks and hence need to derive the strong formulation of the adjoint PDE first.
After training, the neural network is then projected into the FEM ansatz function space of the adjoint problem. Finally, the a posteriori estimates can be made as usual with the DWR method following again Algorithm \ref{dwr_algo}. 

\begin{remark}
     For linear goal functionals the Riesz representation theorem yields the existence and uniqueness of the strong formulation. Nevertheless, deriving the strong form of the adjoint PDE might be very involved for complicated PDEs, such as fluid structure interaction, e.g. \cite{Ri12_dwr,ZeeBrummelenAkkermanBorst2011}, and goal functionals $J: U \rightarrow \mathbb{R}$. In future works, we aim to extend to alternative approaches which do not require the derivation of the strong form.
\end{remark}

\begin{remark}
     We use neural networks to trade off accuracy for speed. In general the neural network approach requires less collocation points than the finite element method. Therefore, we would expect the neural networks to be faster than the finite element method on finer grids. In our numerical tests we used the coordinates of the degrees of freedom as our collocation points, but the collocation points could also be sampled randomly or one could adaptively choose the collocation points as proposed in \cite{deepxde2019}.
\end{remark}

\section{Algorithmic realization}
\label{sec:algo}
In this section, we describe our final algorithm for the neural network guided dual weighted residual method. In the algorithm, we work with hierarchical FEM spaces, i.e. $U_h^{l} \subset U_h^{l,(2)}$ and $V_h^{l} \subset V_h^{l,(2)}$.

\begin{algorithm}[H]
\caption{Neural network guided DWR algorithm}
\begin{algorithmic}[1]
\State Start with some initial guess $u_h^0, l = 1$.
\State Solve the primal problem: Find $u_h^l \in U_h^l$ such that 
\begin{align*}
    \mathcal{A}\left(u_h^l\right)\left(\phi_h^l\right) = 0 \quad \forall \phi_h^l \in V_h^l ,
\end{align*}
using some nonlinear solver.
\State Compute the interpolations $u_h^{l,(2)} = I_h^{(2)}u_h^l \in U_h^{l,(2)}$.
\State Solve the adjoint problem with a neural network: Find $z = d_{\partial\Omega} \cdot z_{NN} + \tilde{g} \in H^1(\Omega)$ such that
\begin{align*}
    z_{NN} = \underset{\hat{z}_{NN}}{\argmin}\,  L(d_{\partial\Omega} \cdot \hat{z}_{NN} + \tilde{g}).
\end{align*}
\State Project the neural network solution in the enriched FEM space
\begin{align*}
    z_h^{l,(2)} = \pi(z)
\end{align*}
with a projection $\pi: H^1(\Omega) \rightarrow V_h^{l,(2)}$.
\State Compute the restriction $z_h^l = i_h z_h^{l,(2)} \in V_h^l$.
\State Compute the error estimator $\eta^{(2)}_h$ defined in (\ref{comp_erroresti}).

\If{$\left| \eta^{(2)}_h \right| < TOL$}
    \State Algorithm terminates with final output $J\left(u_h^l\right)$.
\EndIf
\State Localize error estimator $\eta^{(2)}_h$ and mark elements.
\State Refine marked elements: $\mathbb{T}_h^l \mapsto \mathbb{T}_h^{l+1}, l = l+1.$
\State Go to Step 2.
\end{algorithmic}
\end{algorithm}

\noindent Here we only consider the Galerkin method for which the ansatz function space and the trial function space coincide, i.e. $U=V$, but $U \neq V$ can be realized in a similar fashion. The novelty compared to the DWR method presented in Chapter 2 are step 4 and step 5 of the algorithm. In the following, we describe these parts in more detail. \\

\noindent In step 4, we solve the strong form of the adjoint problem, which for nonlinear PDEs or nonlinear goal functionals also depends on the primal solution $u_h^{l,(2)}$. The strong form of the adjoint problem is of the form (\ref{strong_form}) and thus we can find a neural network based solution by minimizing the loss (\ref{loss}) with L-BFGS \cite{Liu1989}, a quasi-Newton method. We observed that by using L-BFGS sometimes the loss exploded or the optimizer got stuck at a saddle point. Consequently, we restarted the training loop with a new neural network when the loss exploded or used a few steps with the Adam optimizer \cite{kingma2017adam} when a saddle point was reached. Afterwards, L-BFGS can be used as an optimizer again. During training we used the coordinates of the degrees of freedom as our collocation points. We stopped the training when the loss did not decrease by more than $TOL = 10^{-8}$ in the last $n = 5$ epochs or when we reached the maximum number of epochs, which we chose to be $400$. An alternative stopping criterion on fine meshes could be early stopping, where the collocation points are being split into a training and a validation set and the training stops when the loss on the validation set starts deviating from the loss on the training set, i.e. when the neural network begins to overfit on the training data. \\

\noindent In step 5, we projected the neural network based solution into the enriched FEM space by evaluating it at the coordinates of the degrees of freedom, which yields a unique function $z_h^{l,(2)}$.

\section{Numerical experiments}
\label{sec:tests}
In this section we consider two stationary problems (with in total four numerical tests) with our proposed approach. We consider both linear and nonlinear PDEs and goal functionals. The primal problem, i.e. the original PDE, is being solved with bilinear shape functions. 
The adjoint PDE is solved by minimizing the residual of our neural network ansatz (Sections \ref{sec:NN} and \ref{sec:algo}) and we project the solution into the biquadratic finite element space. For studying the performance, we also compute the adjoint problem with finite elements employing biquadratic shape functions. Finally, this neural network solution is being plugged into the PU DWR error estimator (\ref{pu_dwr_estimator}), which decides which elements will be marked for refinement. To realize the numerical experiments, we couple deal.II \cite{dealII92} with LibTorch, the PyTorch C++ API \cite{paszke2019pytorch}.

\subsection{Poisson's equation}
At first we consider the two dimensional Poisson equation with homogeneous Dirichlet conditions on the unit square.
In our ansatz (\ref{ansatz}), we choose the function 
\begin{align*}
    d_{\partial \Omega}(x,y) = x(1-x)y(1-y).
\end{align*}
Poisson's problem is given by
\begin{align*}
    -\Delta u &= f \quad \text{in}\ \Omega := (0,1)^2 \\
    u &= 0 \quad \text{on}\ \partial\Omega.
\end{align*}
For a linear goal functional $J: V \rightarrow \mathbb{R}$ the adjoint problem then reads:

Find $z \in H^1_0(\Omega)$ such that 
\begin{align*}
	( \nabla \psi, \nabla z ) = J(\psi) \quad \forall \psi \in H^1_0(\Omega).
\end{align*}
Here $(\cdot,\cdot)$ denotes the $L^2$ inner product, i.e. $(f,g) := \int_{\Omega}\,f\cdot g\ \mathrm{d}x$.

\subsubsection{Mean value goal functional}
As a first numerical example of a linear goal functional, we consider the mean value goal functional 
\begin{align*}
    J(u) = \frac{1}{|\Omega |}\int_\Omega u\ \mathrm{d}x.
\end{align*}
The adjoint PDE can be written as
\begin{align*}
	( \nabla \psi, \nabla z ) = \left(\psi, \frac{1}{|\Omega |}\right)
\end{align*}
and can be transformed into its strong form
\begin{align*}
    -\Delta z &= \frac{1}{|\Omega |} \quad \text{in}\ \Omega \\
    z &= 0 \qquad \text{on}\ \partial\Omega.
\end{align*}
We trained a fully connected neural network with two hidden layers with 32 neurons each and the hyperbolic tangent activation function for 400 epochs on 1,000 uniformly sampled points. In \cite{raissi2017pinns} it has been shown that wider and deeper neural networks can achieve a lower $L^2$ error between the neural network and the analytical solutions. However, if we use the support points of the FEM mesh as the collocation points, we cannot use bigger neural networks, since we do not have enough training data. Therefore, we decided to use smaller networks.\\

\noindent We compared our neural network based error estimator with a standard finite element based error estimator:

\begin{figure}[H]
\begin{center}
\begin{tabular}{|c|r|r||c|c|c|c|}
    \hline
   \multicolumn{3}{|c||}{} & \multicolumn{2}{|c|}{Est. error} & \multicolumn{2}{|c|}{$I_{eff}$}  \\ \hline
     Ref. & DoFs & $J(u) - J(u_h)$ & FEM & NN & FEM & NN \\ \hline
     0 &9   & 1.17e-2 &  1.15e-2  &  1.14e-2 & 0.979  &  0.971  \\
     1 &25  & 3.17e-3 & 3.14e-3  &  3.14e-3 & 0.992  &  0.990 \\
     2 &81  & 8.10e-4 & 8.08e-4  &  8.08e-4 & 0.998  &  0.998 \\
     3 &289 & 2.03-4 & 2.04e-4  &  2.04e-4 & 1.00  &  1.00 \\
     4 &1089 & 5.03e-5 & 5.11e-5  &  5.11e-5 & 1.02  &  1.02 \\
     5 &4225 & 1.20e-5 & 1.28e-5  &  1.28e-5 & 1.07  &  1.07 \\ \hline
\end{tabular}
\caption{Section \thesubsubsection: Error estimator results for mean value goal functional.}
\end{center}
\end{figure}
\noindent In this numerical test the neural network refined in the same way as the finite element method and both error indicators yield effectivity indices $I_{eff}$ of approximately $1.0$, which means that the exact error and the estimated error were almost identical.
The error reduction is of second order as to be expected and the overall results 
confirm well similar computations presented in \cite{RiWi15_dwr}[Table 1].

\subsubsection{Regional mean value goal functional}
In the second numerical example, we analyze the mean value goal functional which is only being computed on a subset  $D \subset \Omega$ of the domain. We choose $D := \left[0,\frac{1}{4}\right] \times \left[0,\frac{1}{4}\right]$. For the regional goal function
\begin{align*}
    J(u) = \frac{1}{|D|}\int_D u\ \mathrm{d}x
\end{align*}
the strong form of the PDE is given by
\begin{align*}
    -\Delta z &= \frac{\mathbbm{1}_D}{|D|},
\end{align*}
where $\mathbbm{1}_D$ is the indicator function of $D$. The rest of the training setup is the same as for the previous goal functional.\\

\noindent We obtain the following computational results when comparing finite elements with our neural network approach:

\begin{figure}[H]
\begin{center}
\begin{tabular}{|c||r|c|c|c|r|c|c|c|}
    \hline
   & \multicolumn{4}{|c|}{FEM} & \multicolumn{4}{|c|}{NN}  \\ \hline
    Ref. & DoFs & $J(u) - J(u_h)$ & Est. error & $I_{eff}$  & DoFs & $J(u) - J(u_h)$ & Est. error & $I_{eff}$  \\ \hline
     0 &25   & 3.60e-3  &  3.57e-3  & 0.991  &  25   &  3.60e-3  &  3.43e-3 &  0.953\\
     1 &41  & 1.05e-3  &  1.15e-3  & 1.10  &  41   &  1.05e-3   &  1.07e-3 &  1.02\\
     2 &137  & 2.57e-4  &  2.72e-4  & 1.06  &  137   &  2.57e-4  &  2.53e-4 &  0.986\\
     3 &377 & 6.07e-5  &  6.30e-5  & 1.04  &  349  &  6.08e-5  &  5.93e-5 &  0.976\\
     4 &1153 & 1.67e-5  &  1.86e-5  & 1.11  &  1139  &  1.68e-5  &  1.76e-5 &  1.05\\
     5 &3705 & 4.18e-6  &  4.90e-6  & 1.17  &  3635  &  4.40e-6  &  4.92e-6 &  1.12\\ \hline
\end{tabular}
\caption{Section \thesubsubsection: Error estimator results for regional mean value goal functional.}
\end{center}
\end{figure}

\noindent In this example the finite element method and our approach end up with different grid refinements but both had a similar performance and both error indicators had an effectivity index $I_{eff}$ of approximately $1.0$.

\begin{figure}[H]%
    \centering
    \subfloat[FEM]{
    \includegraphics[width=0.45\columnwidth]{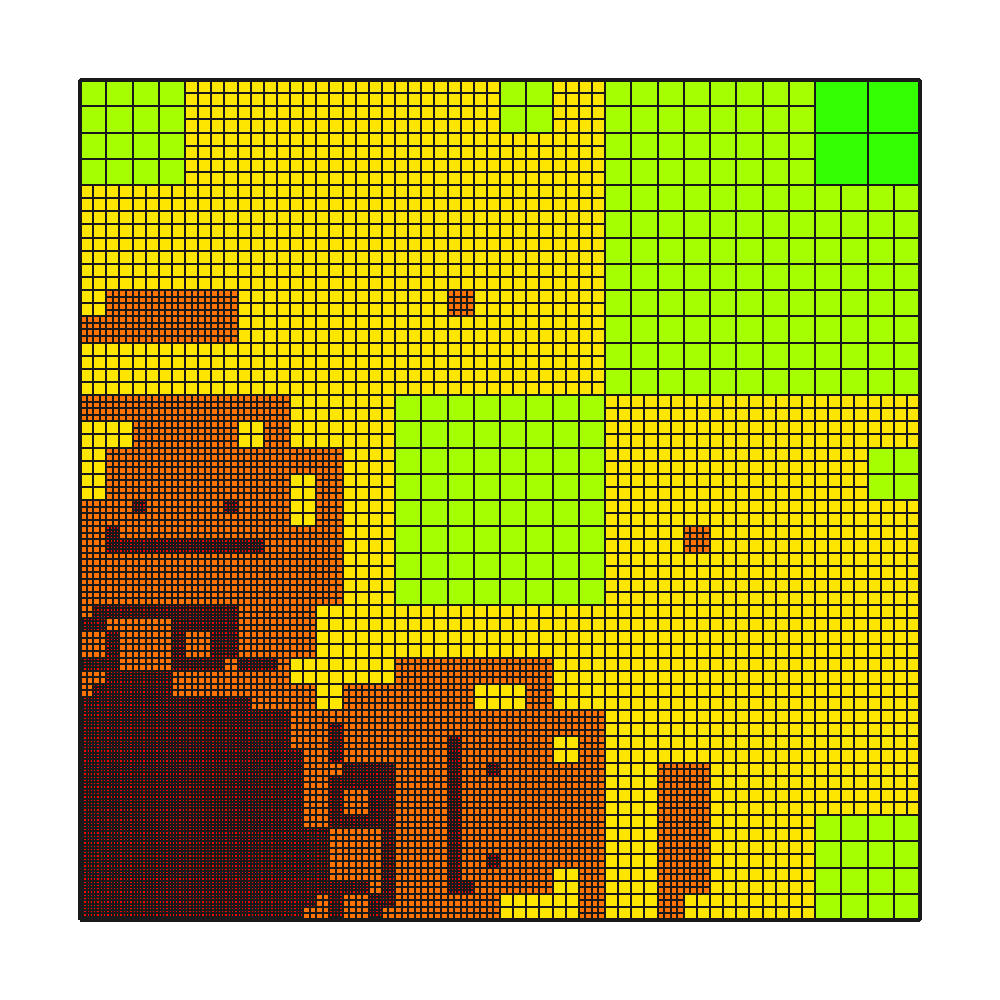}}%
    \quad
    \subfloat[NN]{
    \includegraphics[width=0.45\columnwidth]{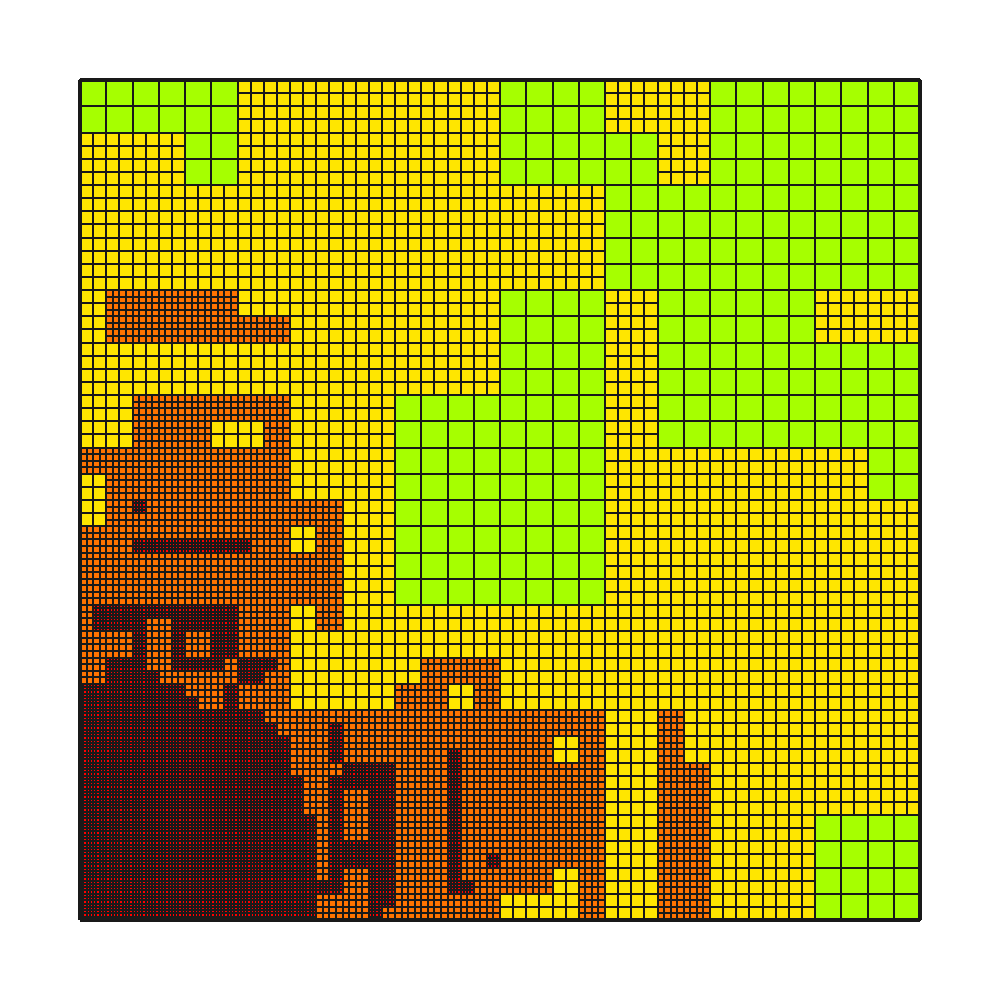}}%
    \caption{Section \thesubsection: Grid refinement with regional mean value goal functional.}%
\end{figure}

\noindent On these grids which have been refined with the different approaches, we can see that the finite element method creates a symmetrical grid refinement. This symmetry can not be observed in the neural network based refinement. Furthermore, our approach refined a few more elements than FEM, but overall our methodology still produced a reasonable grid adaptivity.

\subsubsection{Mean squared value goal functional}
In this third numerical test, an example of a nonlinear goal functional 
is the mean squared value, which reads
\begin{align*}
    J(u) = \frac{1}{|\Omega |}\int_\Omega u^2\ \mathrm{d}x.
\end{align*}

\noindent For a nonlinear goal functional the adjoint problem then needs to be modified to (see also (\ref{adjointproblem}) in Section 2): Find $z \in H^1_0(\Omega)$ such that 
\begin{align*}
	( \nabla \psi, \nabla z ) = J'(u)(\psi) \quad \forall \psi \in H^1_0(\Omega).
\end{align*}

\noindent Computing the Fr\'echet derivative of the mean squared value goal functional, we can rewrite the adjoint problem as

\begin{align*}
	( \nabla \psi, \nabla z ) = \left(\psi, \frac{2u}{|\Omega |}\right)
\end{align*}
and can be transformed into its strong form
\begin{align*}
    -\Delta z &= \frac{2u}{|\Omega |}.
\end{align*}

\noindent Our training setup also changed slightly. The problem statement has become more difficult and we decided to use slightly bigger networks to compute a sufficiently good solution of the adjoint solution. We used three hidden layers with 32 neurons and retrained the neural network on each grid, since the primal solution is part of the adjoint PDE.
\begin{figure}[H]
\begin{center}
\begin{tabular}{|c||r|c|c|c|r|c|c|c|}
    \hline
   & \multicolumn{4}{|c|}{FEM} & \multicolumn{4}{|c|}{NN}  \\ \hline
    Ref. & DoFs & $J(u) - J(u_h)$ & Est. error & $I_{eff}$ & DoFs & $J(u) - J(u_h)$ & Est. error  & $I_{eff}$  \\ \hline
     0 &9   & 7.26e-4  &  5.63e-4  & 0.776  &  9   &  7.26e-4  &  1.89e-4 &  0.261\\
     1 &25  & 1.87e-4  &  1.75e-4  & 0.936  &  25   &  1.87e-4  &  1.34e-4 &  0.713 \\
     2 &81  & 4.71e-5  &  4.64e-5  & 0.987  &  81   &  4.71e-5  &  3.39e-5 &  0.721\\
     3 &289 & 1.16e-5  &  1.18e-5  & 1.01  &  289  &  1.16e-5  &  8.53e-6 &  0.732\\
     4 &1041 & 2.89e-6  &  3.16e-6  & 1.09  &  745  &  4.57e-6  &  3.06e-6 &  0.669\\
     5 &3561 & 6.86e-7  &  9.38e-7  & 1.37  &  2865  &  9.95e-7  &  7.51e-7 &  0.755\\ \hline
\end{tabular}
\caption{Section \thesubsection: Error estimator results for mean squared value goal functional.}
\end{center}
\end{figure}

\noindent Our neural network approach consistently underestimates the error and produces slightly worse results than the FEM solution. Nevertheless, the effectivity index is still sufficiently close to 1 and the grid refinement looks reasonable. Moreover as in the other previous tests, the effecitivity indices $I_{eff}$ are stable without major oscillations.
\begin{figure}[H]%
    \centering
    \subfloat[FEM]{
    \includegraphics[width=0.45\columnwidth]{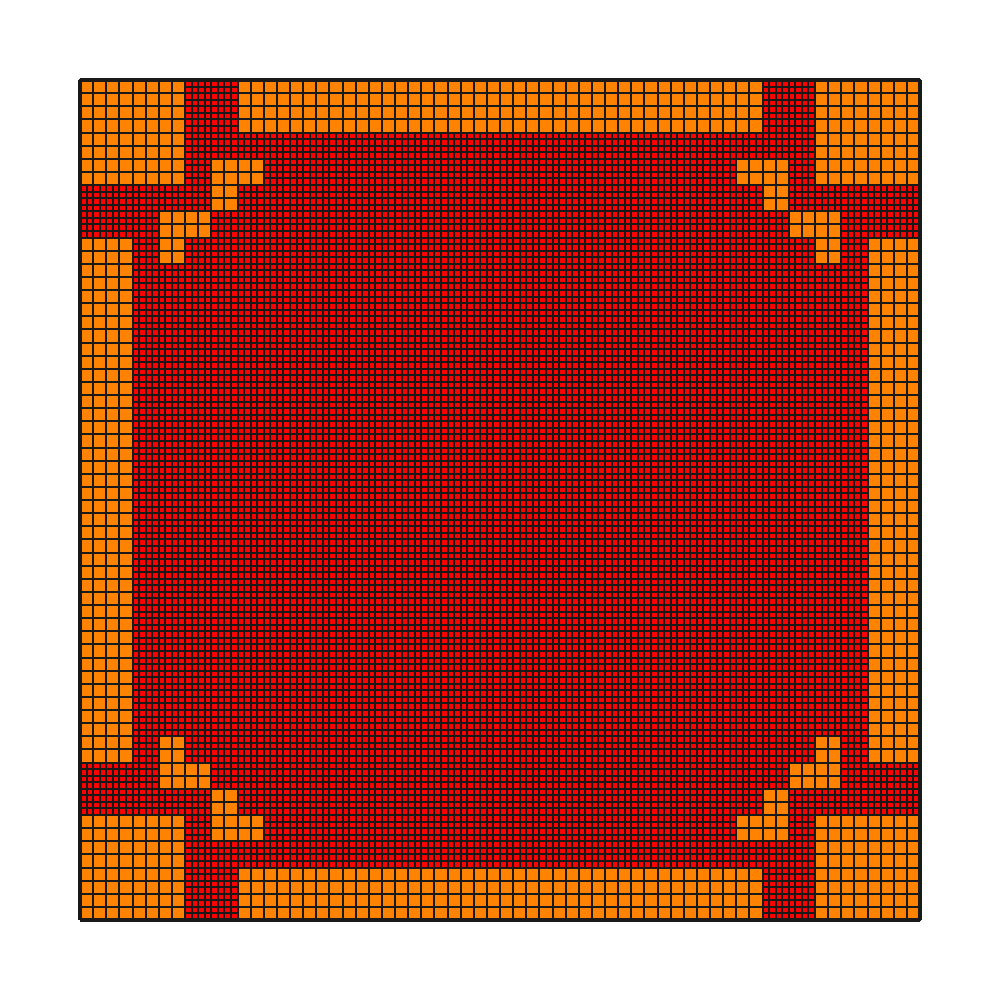}}%
    \quad
    \subfloat[NN]{
    \includegraphics[width=0.45\columnwidth]{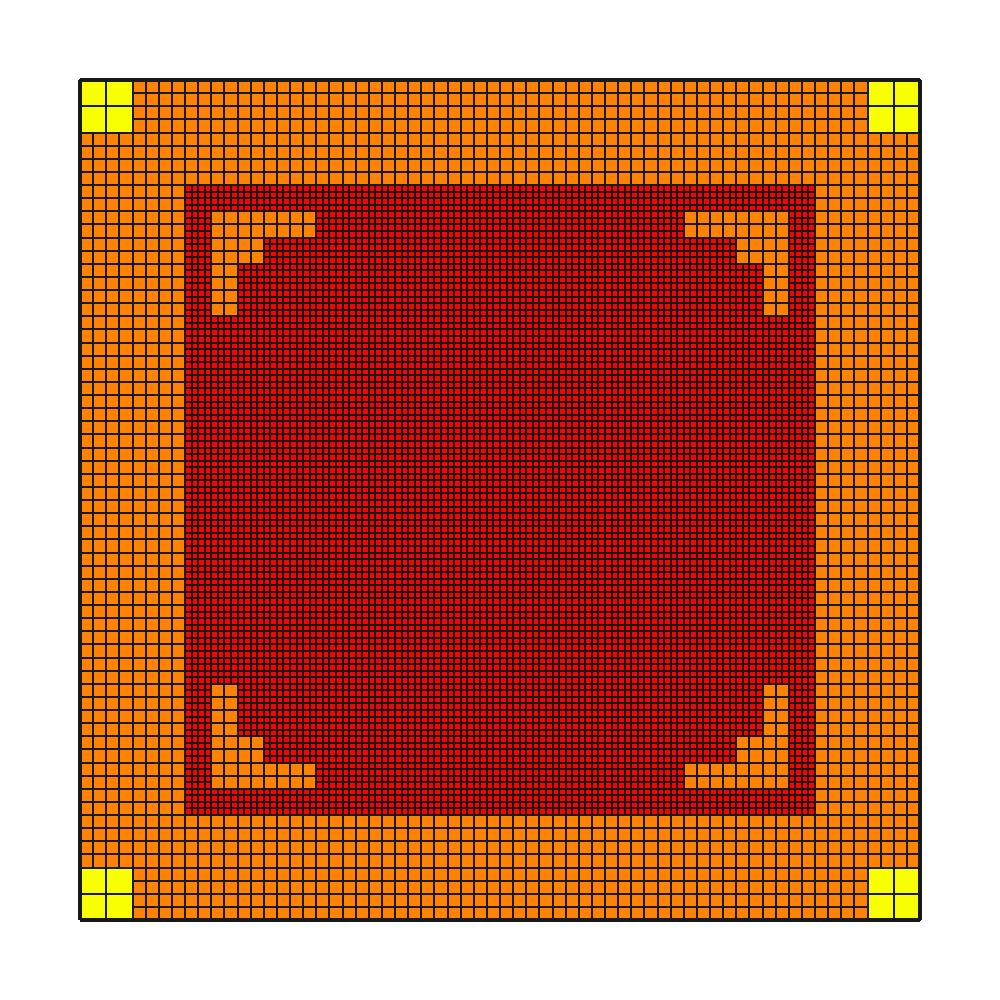}}%
    \caption{Section \thesubsection: Grid refinement with mean squared value goal functional.}%
\end{figure}

\subsection{Nonlinear PDE and nonlinear goal functional}
In the second numerical problem, we now consider the case were both the PDE and the goal functional are nonlinear. We add the scaled nonlinear term $u^2$ to the previous equation, such that the new problem is given by
\begin{align*}
    -\Delta u + \gamma u^2 &= f \quad \text{in}\ \Omega\\
    u &= 0 \quad \text{on}\ \partial\Omega,
\end{align*}
with $\gamma > 0$.
For our nonlinear goal functional, we choose the mean squared value goal functional from the previous example. The adjoint problem thus reads:

Find $z \in H^1_0(\Omega)$ such that 
\begin{align*}
	( \nabla \psi, \nabla z ) + 2\gamma (\psi, zu) = \left(\psi, \frac{2u}{|\Omega |}\right) \quad \forall \psi \in H^1_0(\Omega),
\end{align*}

\noindent with corresponding strong form
\begin{align*}
    -\Delta z + 2\gamma zu&= \frac{2u}{|\Omega |}.
\end{align*}
The training setup is the same as for the previous goal functional. For $\gamma = 50$ we obtain the following results:
\begin{figure}[H]
\begin{center}
\begin{tabular}{|c||r|c|c|c|r|c|c|c|}
    \hline
   & \multicolumn{4}{|c|}{FEM} & \multicolumn{4}{|c|}{NN}  \\ \hline
    Ref. & DoFs&  $J(u) - J(u_h)$ & Est. error & $I_{eff}$ & DoFs  & $J(u) - J(u_h)$ & Est. error & $I_{eff}$   \\ \hline
     0 &9   & 1.21e-3  &  8.64e-4  & 0.713  &  9   &  1.21e-3   &  0.821e-4 &  0.677 \\
     1 &25  & 3.58e-4  &  3.32e-4  & 0.926  &  25   &  3.58e-4   &  4.88e-4 &  1.36  \\
     2 &81  & 9.40e-5  &  9.29e-5  & 0.988  &  81   &  9.40e-5   &  4.24e-5 &  0.451 \\
     3 &289 & 2.33e-5  &  2.39e-5  & 1.03  &  241  &  3.01e-5   &  2.68e-5 &  0.890 \\
     4 &945 & 6.03e-6  &  7.17e-6  & 1.19  &  809  &  7.67e-6   &  4.97e-6 &  0.648 \\
     5 &3089 & 1.25e-6  &  2.15e-6  & 1.72  &  2947  &  1.52e-6   &  1.58e-6 &  1.04 \\ \hline
\end{tabular}
\caption{Section \thesubsection: Error estimator results for the nonlinear PDE.}
\end{center}
\end{figure}

\noindent Our neural network approach produces different results than the finite element method, but at the efficiency indices and the refined grids we observe that our approach still works well for adaptive mesh refinement.  

\begin{figure}[H]%
    \centering
    \subfloat[FEM]{
    \includegraphics[width=0.45\columnwidth]{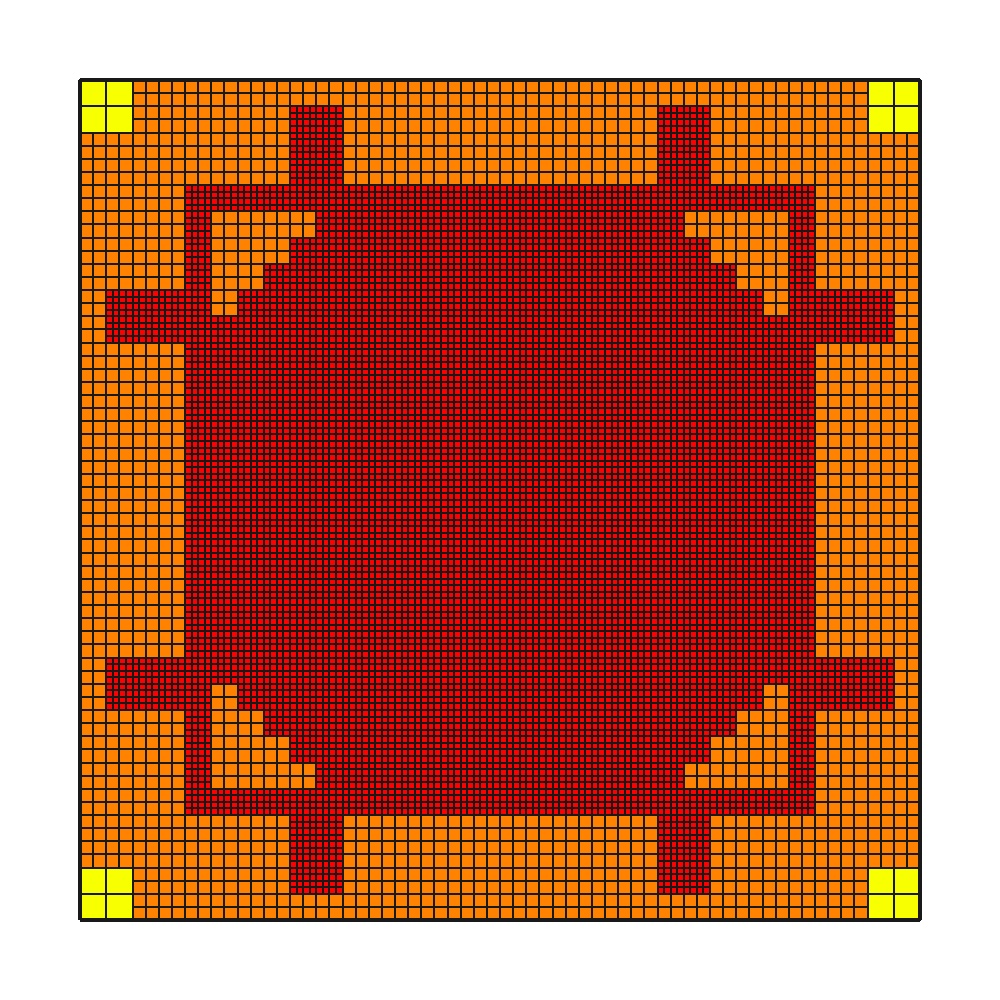}}%
    \quad
    \subfloat[NN]{
    \includegraphics[width=0.45\columnwidth]{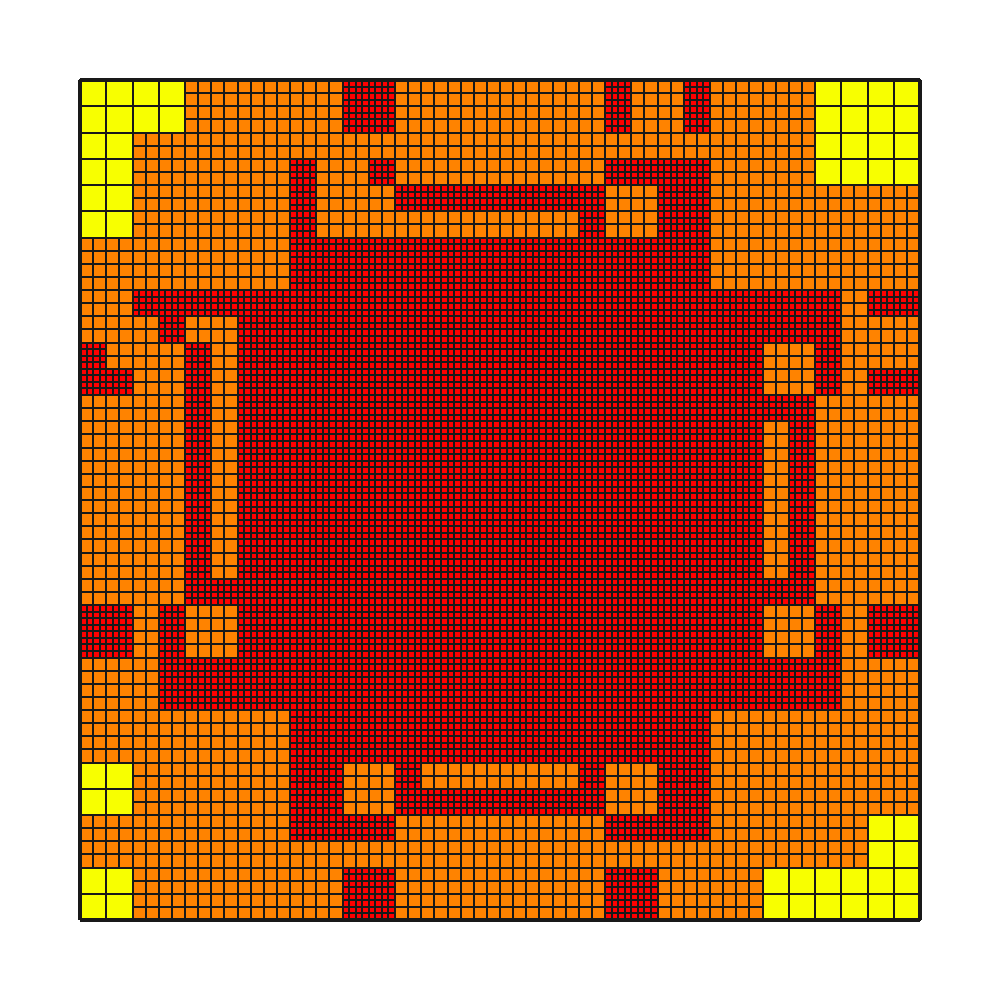}}%
    \caption{Section \thesubsection: Grid refinement for the nonlinear PDE.}%
\end{figure}

\section{Conclusions and outlook}
\label{sec:conclusions}
In this work, we proposed neural network guided a posteriori error estimation with the dual weighted residual method. Specifically, we computed the adjoint solution with feedforward neural networks with two or three hidden layers. To use existing FEM software we first solved the adjoint PDE with neural networks and then projected the solution into the FEM space of the adjoint PDE. We demonstrated experimentally that neural network based solutions of the strong formulation of the adjoint PDE yield excellent approximations for dual weighted residual error estimates. Therefore, neural networks might be an effective way to compute adjoint sensitivities within goal-oriented error estimators for certain problems, when the number of degrees of freedom is high. Furthermore they admit greater flexibility being a meshless method and it would be interesting to investigate in future works how different choices of collocation points influence the quality of the error estimates. A sophisticated choice of collocation points could lead to a significant speedup over the finite element method for a high number of degrees of freedom.
However, an important current limitation of our methodology is that we work with the strong formulation of the PDE, whose derivation from the weak formulation can be very involved for more complex problems, e.g. multiphysics. Hence, if an energy minimization formulation exists, this should be a viable alternative to our strong form of the adjoint PDE. This alternative problem can be solved with neural networks with the "Deep Ritz Method" \cite{E2018, SAMANIEGO2020112790}. Nevertheless, the energy minimization formulation does not exist for all partial differential equations. For this reason in the future, we are going to analyze neural network based methods, which work with the variational formulation, e.g. VPINNs \cite{kharazmi2019variational}.\\

\section*{Acknowledgements}
This work is supported by the Deutsche Forschungsgemeinschaft (DFG) 
under Germany's Excellence Strategy within 
the cluster of Excellence PhoenixD (EXC 2122, Project ID 390833453).

\bibliographystyle{abbrv}

\end{document}